\documentclass[amsthm]{elsart}
\usepackage{yjsco}
\usepackage{amscd,amssymb,amsmath,latexsym,url}
\usepackage[all]{xy}
\usepackage{color,graphicx}
\usepackage{amssymb}
\usepackage{graphicx}%
\newtheorem{theorem}{Theorem}

           {\vspace{3.3mm}
           \noindent{\bf #1}\it}%
           {\vspace{3.3mm}}
\newtheorem{corollary}[theorem]{Corollary}
\newtheorem{lemma}[theorem]{Lemma}
\newtheorem{proposition}[theorem]{Proposition}

\newenvironment{acknowledgement}{\noindent\textbf{Acknowledgements.}}{}

\newcommand{\cp}{\mbox{coeff}}
\def\ov#1{{\overline{#1}}}

\def\Res{{\rm{Res}}}

\def\Sres{{\mbox{Sres}}}

\def\N{{\mathbb N}}

\def\Z{{\mathbb Z}}

\def\Res{{\rm{Res}}}

\begin{document}

\begin{frontmatter}

\title{Sylvester's double sums: \\an inductive proof of the general case }

\thanks{
 Krick was partially
supported by the research projects  UBACyT X-113 2008-2010 and
CONICET PIP 2010-2012   (Argentina);  Szanto's was partially
supported by NSF grant  CCR-0347506 (USA).}

\author{Teresa Krick}
\address{Departamento de Matem\'atica, Facultad de Ciencias Exactas y
Naturales, Universidad de Buenos Aires and IMAS, CONICET,   1428
Buenos Aires, Argentina.}
\ead{krick@dm.uba.ar} \ead[url]{http://mate.dm.uba.ar/\~\,krick}

\author{Agnes Szanto}
\address{Department of Mathematics, North Carolina State
University, Raleigh NC 27695, USA.}
\ead{aszanto@ncsu.edu}
\ead[url]{http://www4.ncsu.edu/\~\,aszanto}



\begin{abstract}
In 1853 J. Sylvester introduced
a {family\/} of double sum expressions
for two finite sets of indeterminates
and showed that some members of the family are essentially
the polynomial subresultants
of the monic polynomials associated with these sets.
In 2009, in a joint work  with C. D'Andrea and H. Hong  we gave the complete description of all the members of the family as expressions in the coefficients of these
polynomials.
In 2010, M.-F. Roy and A. Szpirglas presented a new and natural inductive proof for the cases  considered by Sylvester.
Here we show how induction also allows to obtain the full description of Sylvester's double-sums.
\end{abstract}

\end{frontmatter}

\section{Introduction} Let $A$ and
$B$ be non-empty finite lists (ordered sets) of distinct
indeterminates. In \cite{sylv}, J. Sylvester introduced for each $0\le
p\le |A|$ and $0\le q \le |B|$ the following univariate polynomial in the variable $x$, of degree
$\le p+q$, called the  {\em double-sum expression} in $A$ and $B$:

$$
\operatorname{Sylv}^{p,q}(A,B)  :=
\sum_{\scriptsize{\begin{array}{c} A^{\prime }\subset
A,\,B^{\prime}\subset B
\\[-1.8mm]
|A^{\prime}|=p,\,|B^{\prime}|=q
\end{array}}}R(x,A^{\prime
})\,R(x,B^{\prime})\,\frac{R(A^{\prime},B^{\prime})\,R(A-
A^{\prime},B- B^{\prime})}{R(A^{\prime},A- A^{\prime
})\,R(B^{\prime},B- B^{\prime})},$$ where for sets $Y$, $Z$ of
indeterminates,
\[
R(Y,Z):=\prod_{y\in Y,z\in Z}(y-z), \ \quad R(y,Z):=\prod_{z\in
Z}(y-Z).
\]
and by convention $R(Y,\emptyset)=1$.

\medskip

Let now  $f, g$ be  monic univariate
polynomials such that
$$
f=\prod_{\alpha\in A}(x-\alpha)=x^m+a_{m-1}x^{m-1}+\ldots+a_0 \ \mbox { and } \
g =\prod_{\beta\in B}(x-\beta)=x^n+b_{n-1}x^{n-1}+\ldots+b_0,
$$
where $m:=|A|\geq1$ and $n:=|B|\geq1$.
The {\em $k$-th subresultant} of the
polynomials $f$ and $g$ is defined, for $0\le k <  \min\{m,n\}$ or $k=\min\{m,n\}$ when
$m\neq n$,  as
{
\begin{equation}\label{defsub}
\Sres_k(f,g)
:=\det%
\begin{array}{|cccccc|c}
\multicolumn{6}{c}{\scriptstyle{m+n-2k}}\\
\cline{1-6}
a_{m} & \cdots & & \cdots & a_{k+1-\left(n-k-1\right)}& x^{n-k-1}f(x)&\\
& \ddots & && \vdots  & \vdots &\scriptstyle{n-k}\\
&  &a_{m}&\cdots &a_{k+1}& x^0f(x)& \\
\cline{1-6}
b_{n} &\cdots & & \cdots & b_{k+1-(m-k-1)}&x^{m-k-1}g(x)&\\
&\ddots &&&\vdots & \vdots &\scriptstyle{m-k}\\
&& b_{n} &\cdots & b_{k+1} & x^0g(x)&\\
\cline{1-6} \multicolumn{2}{c}{}
\end{array}
\end{equation}}
with $a_{\ell}=b_{\ell}=0$ for $\ell<0$. For $k=0$,  $ \Sres_0(f,g)$
coincides with the resultant:
\begin{equation}\label{result}\Res(f,g)= \prod_{\alpha\in A}
g(\alpha) = (-1)^{mn}  \prod_{\beta\in B} f(\beta).\end{equation}

Also, for instance,
\begin{equation}\label{sresm0} \Sres_m(f,g)=f \mbox{ for } m<n \ \mbox{ and } \ \Sres_n(f,g)=g \mbox{ for } n<m. \end{equation}

\medskip
Relating Sylsvester's double sums with the polynomials $f$ and $g$, it is immediate that
\begin{equation}\label{00} \operatorname{Sylv}^{0,0}(A,B)=R(A,B)=\Res(f,g),\end{equation}
\begin{equation}\label{m0} \operatorname{Sylv}^{m,0}(A,B)=R(x,A)=f \ \mbox{ and } \  \operatorname{Sylv}^{0,n}(A,B)=R(x,B)=g,\quad \end{equation}
\begin{equation}\label{mn}\operatorname{Sylv}^{m,n}(A,B)=R(x,A)\,R(x,B)\,R(A,B)=\Res(f,g)\,f\,g.\end{equation}

\medskip
More generally, every value of the polynomial  $\operatorname{Sylv}^{p,q}(A,B)$,  which is symmetric in the $\alpha$'s
and in the $\beta$'s,
can be expressed as a polynomial in $x$ whose coefficients are
rational functions in the $a_i$'s and the $b_j$'s.
Sylvester  in \cite{sylv} gave
this rational expression for the
following values of $(p,q)$:

\smallskip
\begin{enumerate}
\item If $0\le k:=p+q<\min\{m,n\}$ or if  $k=m<n$, then \cite[Art.~21]{sylv}:
$$\operatorname*{Sylv}\nolimits^{p,q}(A,B)=(-1)^{p(m-k)}{k\choose p}\,
\Sres_k(f,g).$$

\item If $p+q=m=n$, then \cite[Art.~22]{sylv}:
$$\operatorname*{Sylv}\nolimits^{p,q}(A,B)={m-1\choose q}\,f+{m-1\choose p}\,g.$$
 \item If $m<p+q<n-1,$ then \cite[Arts.~23~\&~24]{sylv}):
$$\operatorname*{Sylv}\nolimits^{p,q}(A,B)=0.$$
\item If $m<p+q=n-1$, then \cite[Art. 25]{sylv}:
$\operatorname*{Sylv}\nolimits^{p,q}(A,B)$ is a ``numerical
multiplier'' of $f$, but the ratio
is not established.
\end{enumerate}

\medskip
In \cite[Th.0.1 and Prop. 2.9]{LP},  A. Lascoux and P. Pragacz
presented  new proofs for the cases covered by  Items (1) and (2).
More recently, in a joint work with C. D'Andrea and H. Hong in
\cite[Th.2.10]{DHKS09} we introduced   a unified matrix formulation
that  allowed us to give an explicit formula for all possible values
of $(p,q)$, i.e. for $0\le p\le m, 0\le q\le n$. The proofs there
were elementary though cumbersome. In 2010, M.-F. Roy and A.
Szpirglas, were able to produce in \cite[Main theorem]{RoSp} a new
and natural inductive proof also for the cases covered by Item (1)
and (2). The aim of this note is to give, inspired by \cite{RoSp},
an elementary inductive proof for all the cases. We furthermore show
how the cases $p+q>\min\{m,n\}$, which seem somehow less natural
since there is no ``natural"  expression associated to them (and
were therefore not previously considered by Lascoux and Pragasz and
Roy and Szpirglas) immediately yield simple proofs for other known
interesting cases, as for instance for the cases $p+q=m<n $ and
$p+q=m=n$, which didn't have simple proofs yet.

\medskip
Let us now introduce the necessary notation  to formulate our main result.\\
As in \cite{DHKS09}, we split the last column of the matrix in \eqref{defsub} to write $\Sres_k(f,g)$ as the sum of two determinants, obtaining an
expression
\begin{equation}\label{subres}
\Sres_k(f,g)=F_k (f,g)\, f+G_k(f,g)\,g
\end{equation}
where the polynomials $F_k(f,g)$ and $G_k(f,g)$  are defined for
$0\le k <  \min\{m,n\}$ or $k=\min\{m,n\}$ when $m\ne n$
 as the determinants of the $(m+n-2k)$-matrices:

{\tiny $$
  F_k(f,g)
:= \det
\begin{array}{|cccccc|c}
\cline{1-6}
a_{m} & \cdots & & \cdots & a_{k+1-\left(n-k-1\right)}& x^{n-k-1}&\\
& \ddots & && \vdots  & \vdots&\scriptstyle n-k \\
&  &a_{m}&\cdots &a_{k+1}& x^0 & \\
\cline{1-6}
b_{n} &\cdots & & \cdots & b_{k+1-(m-k-1)}&0 &\\
&\ddots &&&\vdots & \vdots & \scriptstyle m-k \\
&& b_{n} &\cdots & b_{k+1} & 0 &\\
\cline{1-6}
\end{array}\ , \
G_{k}(f.g)
:=\det%
\begin{array}{|cccccc|c}
\cline{1-6}
a_{m} & \cdots & & \cdots & a_{k+1-\left(n-k-1\right)}& 0&\\
& \ddots & && \vdots  & \vdots &\scriptstyle n-k \\
&  &a_{m}&\cdots &a_{k+1}& 0 & \\
\cline{1-6}
b_{n} &\cdots & & \cdots & b_{k+1-(m-k-1)}&x^{m-k-1}&\\
&\ddots &&&\vdots & \vdots &\scriptstyle m-k \\
&& b_{n} &\cdots & b_{k+1} & x^0&\\
\cline{1-6}
\end{array}.
 $$}
We observe that when $k<\min\{m,n\}$, $\deg
F_k(f,g)\le   n-k-1$ and $\deg G_k(f,g)\le  m-k-1$. Also
\begin{equation}\label{Fm}F_m(f,g)= 1 ,\  \ G_m(f,g)=0 \ \mbox{ for } m<n \ \mbox{ and } \
F_n(f,g)= 0 , \ G_n(f,g)=1 \ \mbox{ for } n<m \end{equation}
\begin{equation}\label{Gm-1}G_{m-1}(f,g)=1 \mbox{ for } m\le n \ \mbox{ and } \ F_{n-1}(f,g)=(-1)^{m-n+1} \mbox{ for } n\le m.\end{equation}

 \medskip
We finally introduce the following notation that we
 will keep all along in this text.
 Given $m,n\in \N$,  $p,q\in \Z$ such that $0\le p\le m$, $0\le q\le
 n$ and $k=p+q$, we set
  $$\ov p:=m-p,\ \ \ov q:= n-q \ \mbox{ and } \ \ov k:=\ov p + \ov
  q-1
= m+n- k-1.$$ Sylvester's double sums, for $k$ ``too big"  w.r.t.
$m$ and $n$, will be expressed in our result in terms of the
polynomials $F_{\ov k}(f,g)$ and $G_{\ov k}(f,g)$, well-defined
since  the condition $n-1\le k\le m+n-1$ for $m<n$ is equivalent to
$0\le \ov k\le m$, and the condition $m\le k\le 2m-1$ for $m=n$ is
equivalent to $0\le \ov k\le m-1$.

\begin{theorem}\label{theorem} (See also \cite[Th.2.10]{DHKS09}) \\Set $1\le m\le n$, and let $0\le p\le m$, $0\le q\le
 n$ and $k=p+q$.  \\
Then, for  $(p,q)\ne (m,n)$,\\
-- when $m<n$: {\small $$\operatorname{Sylv}^{p,q}(A,B)
=\left\{\begin{array}{lll}(-1)^{p(m-k)}{k\choose p}\Sres_k(f,g)&\mbox{for}& 0\le k\le m\\
0&\mbox{for}& m+1\le k\le n-2\ \mbox{when }  m\le n-3\\
(-1)^{c}\Big({\ov k \choose \ov p}F_{\ov k}(f,g)\,f - {\ov k\choose
\ov q}G_{\ov k}(f,g)\,g\Big)&\mbox{for}& n-1\le k\le m+n-1
\end{array}\right.$$}
-- when $m=n$: {\small $$\operatorname{Sylv}^{p,q}(A,B)
=\left\{\begin{array}{lll}(-1)^{p(m-k)}{k\choose p}\Sres_k(f,g)&\mbox{for}& 0\le k\le m-1\\
(-1)^{c}\Big({\ov k\choose \ov p}F_{\ov k}(f,g)\,f - {\ov k\choose
\ov q}G_{\ov k}(f,g)\,g\Big)&\mbox{for}& m\le k\le 2m-1
\end{array}\right. ,$$}
where $c:=\ov p \,\ov q+ n-p-1+nq$;\\
and for  $(p,q)=(m,n)$:
$$\operatorname{Sylv}^{m,n}(A,B)=\Res(f,g)\,f\,g.$$
 \end{theorem}

Theorem~\ref{theorem} can be written in a more uniform manner
instead of being split in cases: by Identity~\eqref{subres},
 for $0\le k\le m$  when $m<n$ and for $0\le k<m$
when $m=n$,
$$\operatorname{Sylv}^{p,q}(A,B)
=(-1)^{p(m-k)} \Big({k\choose p}F_k(f,g)f + {k\choose
q}G_k(f,g)g\Big),$$ or for $0\le \ov k\le m$  when $m<n$  and for
$0\le \ov k<m$, when $m=n$,
\begin{multline}\label{eqex}\operatorname{Sylv}^{p,q}(A,B)
= (-1)^c\Big( {\ov k\choose \ov p}\Sres_{\ov k}(f,g)-{\ov k+1\choose
\ov q}G_{\ov k }(f,g)\,g\Big)\\  = (-1)^c\Big( {\ov k+1\choose \ov
p}F_{\ov k}(f,g)\,f-{\ov k \choose \ov q}\Sres_{\ov k}(f,g)\Big).
\end{multline}

The cases ``in between", for $m+1\le k\le n-2$ when $m\le n-3$, are
the cases when neither $0\le k\le m$ nor $0\le \ov k\le m$, i.e. the
cases when the corresponding matrices $F_k$, $G_k$ and $F_{\ov k}$,
$G_{\ov k}$ are not defined (or could be defined as $0$ for
uniformity).

 \medskip
We also note that the  case $k=m=n-1$ is covered twice:
$\Sres_m(f,g)=f= F_m(f,g)f-G_m(f,g)g$ since $\ov k =m$, $F_m=1$ and
$G_m=0$. Finally the case $p=m$, $q=n$ is Identity~\eqref{mn}.

\medskip The proof of Theorem \ref{theorem} is based, as the proof in \cite{RoSp}, on specialization
properties.

\section{Specialization properties}
The
following specialization  property of  Sylvester's double sums is
well-known and proved in \cite[Lemma 2.8]{LP}. It is also reproved in
\cite[Prop.3.1]{RoSp}, where it is used as one of the key ingredients of their
inductive proof for the cases $k\le m<n$ and $k<m=n$. We repeat it here for sake of
completeness.

\begin{lemma}\label{sylvester} For any  $\alpha \in A$  and $\beta\in
B$,
\begin{align*}
\bullet \ & \operatorname{Sylv}^{p,q}(A,B) (\alpha) =
(-1)^{p}\,\cp_{p+q}\big(\operatorname{Sylv}^{p,q}(A-\alpha,B)\big)\,R(\alpha,B)
\ \mbox{ for } \ 0\le p<m \mbox{ and } 0\le q\le n,\\ \bullet \ &
\operatorname{Sylv}^{p,q}(A,B) (\beta) = (-1)^{q+\ov
p}\,\cp_{p+q}\big(\operatorname{Sylv}^{p,q}(A,B-\beta)\big)\,R(\beta,A)
\ \mbox{ for }  \ 0\le p\le m  \mbox{ and } 0\le q<n.
\end{align*}
Here $\cp_{p+q}$ denotes the
coefficient of order  $p+q$ of
$\operatorname{Sylv}^{p,q}(A,B-\beta)$.

\end{lemma}
\begin{proof}
\begin{align*} \operatorname{Sylv}^{p,q}(A,B)  (\alpha)&=
\sum_{\scriptsize{\begin{array}{c} A^{\prime }\subset
A-\alpha,\,B^{\prime}\subset B
\\[-1.8mm]
|A^{\prime}|=p,\,|B^{\prime}|=q
\end{array}}}R(\alpha,A^{\prime
})\,R(\alpha,B^{\prime})\,\frac{R(A^{\prime},B^{\prime})\,R(A-
A^{\prime},B- B^{\prime})}{R(A^{\prime},A- A^{\prime
})\,R(B^{\prime},B- B^{\prime})}\\
&= (-1)^{p}\,R(\alpha, B)\sum_{\scriptsize{\begin{array}{c}
A^{\prime }\subset A-\alpha,\,B^{\prime}\subset B
\\[-1.8mm]
|A^{\prime}|=p,\,|B^{\prime}|=q
\end{array}}}\frac{R(A^{\prime},B^{\prime})\,R((A-\alpha)-
A^{\prime},B- B^{\prime})}{R(A^{\prime},(A-\alpha)- A^{\prime
})\,R(B^{\prime},B- B^{\prime})}\\
&= (-1)^{p}\,\cp_{p+q}
\big(\operatorname{Sylv}^{p,q}(A-\alpha,B) \big)\,R(\alpha,B).
\end{align*}
\noindent
The second identity is a consequence of the fact that
$$ \operatorname{Sylv}^{p,q}(A,B)=(-1)^{pq}\,(-1)^{\ov p\, \ov q}\, \operatorname{Sylv}^{q,p}(B,A).$$
\end{proof}

In the following we replace the specialization property of
subresultants proved in \cite[Prop. 4.1]{RoSp} by the specialization
property of the polynomials $F_k$ and $G_k$. This will allow a more
uniform and simpler proof of our main theorem, covering all cases of
$p$ and $q$.

\begin{lemma} \label{dos} For any root $\alpha $  of $f$ and  any root $\beta$ of $g$, we have
\begin{align*}\bullet \ &F_k(f,g)(\beta)= - \,\cp_{n-k-1}\big(
F_{k-1}(f,\frac{g}{x-\beta})\big)\  \mbox{ for } \ 1\le k\le
\min\{m,n\}-1,\\ \bullet \ &G_k(f,g)(\alpha) = (-1)^{m-k-1}
\,\cp_{m-k-1}\big( G_{k-1}(\frac{f}{x-\alpha},g)\big) \ \mbox{ for }
\ 1\le k\le \min\{m,n\}-1.
\end{align*}
Here $\cp_{n-k-1}$ (resp. $\cp_{m-k-1}$) denotes  the coefficient of order  $n-k-1$ (resp.
$m-k-1$) of the corresponding polynomial.
\end{lemma}

\begin{proof}
Given a root  $\beta$  of $g$, we set
$$\frac{g}{x-\beta}:=x^{n-1}+b'_{n-2}x^{n-2}+\cdots + b'_0.$$
The following relationship between the coefficients of $g$ and of
$\frac{g}{x-\beta}$ is straightforward:
 \begin{eqnarray}\label{bi}
 b_i=b'_{i-1}-\beta b'_i \  \mbox{for } \ 1\le i\le n-1
 \ \mbox{ and } \ b_0=-\beta b'_0.
 \end{eqnarray}
(Here $b_n=b'_{n-1}=1$.)\\
 First consider
{\scriptsize
\begin{align*}& \cp_{n-k-1}\big(F_{k-1}(f,\frac{g}{x-\beta})\big)=\cp_{n-k-1}\big(\det
\begin{array}{|cccccc|c}
\cline{1-6}
a_{m} & \cdots & & \cdots & a_{k-(n-k-1)}& x^{n-k-1}&\\
& \ddots & && \vdots  & \vdots &\scriptstyle (n-1)-(k-1) \\
&  &a_{m}&\cdots &a_k& x^0 & \\
\cline{1-6}
b'_{n-1} &\cdots & & \cdots &b'_{k-(m-k)}&0&\\
&\ddots &&&\vdots & \vdots &\scriptstyle m-(k-1) \\
&& b'_{n-1} &\cdots &b'_k & 0&\\
\cline{1-6}
\end{array}\big)\\
& \ =(-1)^{m+n}\det \begin{array}{|cccccc|c} \cline{1-6}
0 & a_{m}&\cdots & & \cdots & a_{k-(n-k-2)}&\\
& &\ddots  &&& \vdots   &\scriptstyle n-1-k \\
&  &&a_{m}&\cdots &a_k & \\
\cline{1-6}
b'_{n-1}&\cdots&\cdots & &\cdots  &b'_{k-(m-k)}&\\
&\ddots &&& & \vdots&\scriptstyle m-k+1 \\
&& b'_{n-1} &\cdots & &b'_k&\\
\cline{1-6}
\end{array} \
=\ (-1)^{m-k+1}\det \begin{array}{|cccccc|c} \cline{1-6}
a_{m} &\cdots & & \cdots & a_{k-(n-k-2)}& &\\
& \ddots & &\vdots& &  &\scriptstyle n-1-k \\
 &  &a_{m}&\cdots &a_k && \\
\cline{1-6}
b'_{n-1}&\cdots & & \cdots &b'_{k-(m-k-1)}& &\\
&\ddots &&\vdots & &&\scriptstyle m-k \\
&& b'_{n-1} &\cdots &b'_k&&\\
\cline{1-6}
\end{array}.\\
\end{align*}}
We apply elementary column operations on the matrix above, replacing
the $j$-th column $C_j$ by $C_j-\beta C_{j-1}$ starting from the
last column $C_{n+m-2k-1}$ up to the second column $C_2$, and
using the relations in (\ref{bi}): {\scriptsize
\begin{align}\label{coeff}
 \cp_{n-k-1}\big(F_{k-1}(f,\frac{g}{x-\beta})\big)=(-1)^{m-k+1}\det
\begin{array}{|cllcccc|c} \cline{1-7}
a_{m} &a_{m-1}-\beta a_m&\cdots & & \cdots & a_{k-(n-k-2)}-\beta a_{k+1-(n-k-2)}& &\\
 &\ddots && &\vdots& &  &\scriptstyle n-1-k \\
   &&a_{m}&a_{m-1}-\beta a_m&\cdots &a_k-\beta a_{k+1} && \\
\cline{1-7}
b_{n}&b_{n-1}&\cdots & & \cdots &b_{k+1-(m-k-1)}& &\\
&\ddots &&&\vdots & &&\scriptstyle m-k \\
&& b_{n} &b_{n-1}& \cdots &b_{k+1}&&\\
\cline{1-7}
\end{array}.
\end{align}}
Next  consider {\scriptsize
\begin{align*}&
F_k(f,g)(\beta)= \det
\begin{array}{|cccccc|c}
\cline{1-6}
a_{m} & \cdots & & \cdots & a_{k+1-(n-k-1)}& \beta^{n-k-1}&\\
& \ddots & && \vdots  & \vdots &\scriptstyle n-k\\
&  &a_{m}&\cdots &a_{k+1}& \beta^0 & \\
\cline{1-6}
b_{n} &\cdots & & \cdots &b_{k+1-(m-k-1)}&0&\\
&\ddots &&&\vdots & \vdots &\scriptstyle m-k \\
&& b_{n} &\cdots &b_{k+1}k & 0&\\
\cline{1-6}
\end{array}.
\end{align*}}
We apply elementary row operations on the matrix above, replacing
the $i$-th row $R_i$ by $R_i-\beta R_{i+1}$, starting from the first
row $R_1$ up to row $R_{n-k-1}$: {\scriptsize
\begin{align}\label{Fk} \nonumber
F_k(f,g)(\beta)=\det
\begin{array}{|ccccccc|c}
\cline{1-7}
a_{m} & a_{m-1}-\beta a_m & && \cdots & a_{k+1-(n-k-1)}-\beta a_{k+2-(n-k-1)}& 0&\\
& \ddots && && \vdots  & \vdots &\scriptstyle n-k\\
&  &a_{m}\;\;a_{m-1}-\beta a_m&&\cdots &a_{k+2}-\beta a_{k+1}& 0 & \\
&  &a_{m}&a_{m-1}&\cdots &a_{k+1}& 1& \\
\cline{1-7}
b_{n} &\cdots & && \cdots &b_{k+1-(m-k-1)}&0&\\
&\ddots &&&&\vdots & \vdots &\scriptstyle m-k \\
&& b_{n} &\cdots &&b_{k+1} & 0&\\
\cline{1-7}
\end{array}\\
=(-1)^{m-k}\det\begin{array}{|ccccccc|c} \cline{1-7}
a_{m} & a_{m-1}-\beta a_m & && \cdots & a_{k+1-(n-k-1)}-\beta a_{k+2-(n-k-1)}& &\\
& \ddots && && \vdots  & &\scriptstyle n-k-1\\
&  &a_{m}\;\;a_{m-1}-\beta a_m&&\cdots &a_{k+2}-\beta a_{k+1}& & \\
 \cline{1-7}
b_{n} &\cdots & && \cdots &b_{k+1-(m-k-1)}&&\\
&\ddots &&&&\vdots &  &\scriptstyle m-k \\
&& b_{n} &\cdots &&b_{k+1} & &\\
\cline{1-7}
\end{array} .\end{align}}
We obtain the first identity of the statement by comparing
(\ref{coeff}) and (\ref{Fk}).

\noindent For the second identity, we have
\begin{align*}G_k(f,g)(\alpha)&=(-1)^{(n-k)(m-k)}F_k(g,f)(\alpha)\\ & =(-1)^{(n-k)(m-k)+1}
\cp_{m-k-1}\big(F_{k-1}(g,\frac{f}{x-\alpha})\big)\\ &
=(-1)^{(n-k)(m-k)+1}(-1)^{(m-k)(n-k+1)}\cp_{m-k-1}\big(G_{k-1}(\frac{f}{x-\alpha},g)\big)\\
& =(-1)^{m-k-1}\cp_{m-k-1}\big(G_{k-1} (\frac{f}{x-\alpha},g)\big).
\end{align*}
\end{proof}

As an immediate consequence we obtain  the specialization property
of subresultants which seemed to have been stated and proved for the
first time in \cite[Prop.~4.1]{RoSp}.
\begin{corollary}\label{corrs}  For any root $\alpha $  of $f$, any root $\beta$ of $g$
 and any $0\le k < \min\{m,n\}$, we have
\begin{align*}\bullet \ &\Sres_k(f,g)(\beta)  =(-1)^{m-k}\, \cp_k \big(\Sres_k(f\,
,\,\frac{g}{x-\beta})\big)\,f(\beta), \\ \bullet \ &
\Sres_k(f,g)(\alpha) = \cp_k \big(\Sres_k(\frac{f}{x-\alpha}\,
,\,g)\big)\,g(\alpha).\end{align*} Here $\cp_k$ denotes the
coefficient of order $k$ of the corresponding polynomial.
\end{corollary}
\begin{proof} It is sufficient to prove the first identity, since the second identity is a consequence of
$$\Sres_k(g,f)=(-1)^{(m-k)(n-k)}\Sres_k(f,g).$$
  By \eqref{subres} and
the previous lemma,
$$\Sres_k(f,g)(\beta)=F_k(f,g)(\beta)\,f(\beta)=- \,\cp_{n-k-1}\big(
F_{k-1}(f,\frac{g}{x-\beta})\big)\,f(\beta).$$ Now it is immediate
to verify by the definition of the principal scalar subresultant of
order $k$ that
$$\cp_{n-k-1}\big(F_{k-1}(f,\frac{g}{x-\beta})\big)=(-1)^{m-k-1}\cp_k\big(\Sres_k(f,\frac{g}{x-\beta})\big).$$
\end{proof}

\section{Proof of Theorem~\ref{theorem}}

\noindent It turns out that the cases of Theorem 1 where $k$ is
``big" are easy to prove by induction and will be used later in the
other cases. That is why we start with this case first in the
following proposition. The proof will use a lemma for the extremal
cases $(p,n)$ and $(m,q)$, which is given after the proposition.
  We recall that  $\ov p=m-p$, $ \ov q= n-q$, and $\ov k=m+n-k-1$.

\begin{proposition}\label{casogrande} Set $1\le m\le n$ and let $0\le p\le m,\ 0\le q\le n$ and $k=p+q$ be
 such that $n-1\le k\le m+n-1$, i.e. $0\le \ov k\le m$, when $m<n$ or $m\le k\le 2m-1$, i.e. $0\le \ov k\le m-1$,
 when $m=n$. Then
$$\operatorname{Sylv}^{p,q}(A,B)
= (-1)^{\ov p\,\ov q + n-p-1+nq}\Big({\ov k\choose \ov p}F_{\ov
k}(f,g)\,f - {\ov k\choose \ov q}G_{\ov k}(f,g)\,g\Big).$$
\end{proposition}

\begin{proof} By induction on $\ov k\ge 0$:\\
The case $\ov k=0$ implies $(p,q)=(m-1,n)$ or $(p,q)=(m,n-1)$ and
will follow from Lemma~\ref{teorema1}.

\smallskip
\noindent Now set $\ov k>0$. \\
-- For $p=m$ and $q<n$ or $p<m$ and $q=n$, also by Lemma~\ref{teorema1},
$$\operatorname{Sylv}^{m,q}(A,B)=(-1)^{n-m-1+nq}F_{\ov q-1}(f,g)f \ \mbox{ and } \ \operatorname{Sylv}^{p,n}(A,B)=(-1)^{p}G_{\ov p-1}(f,g)g$$
accordingly, which matches the statement since in these cases ${\ov
k\choose \ov q}$ or ${\ov k \choose \ov p}$ equals $0$.

\noindent -- For $p<m$ and $q<n$, we specialize
$\operatorname{Sylv}^{p,q}(A,B)$ of degree $k\le m+n-2$ in the $m+n$
elements of $A\cup B$ by means of Lemma~\ref{sylvester} and  the
inductive hypothesis:
\begin{align*}
\operatorname{Sylv}^{p,q}(A,B) (\alpha)& =
(-1)^{p}\,\cp_{k}\big(\operatorname{Sylv}^{p,q}(A-\alpha,B)\big)\,g(\alpha)\\
&= (-1)^{c'+p}\,\cp_k\Big({\ov k-1\choose \ov p -1}\Sres_{\ov
k-1}(\frac{f}{x-\alpha},g) -{\ov k\choose \ov q}G_{\ov
k-1}(\frac{f}{x-\alpha},g)g\Big)\,g(\alpha),\end{align*}
by Identity~\eqref{eqex}. Here $c'= (\ov p-1)\ov q+n-p-1+nq$.\\
Note that we are looking for the coefficient of degree $k$ of the
expression between brackets; the condition  $\ov k -1\le m-1<n-1\le
k$ in case $m<n$ and $\ov k -1\le m-2 <k$ in case $m=n$ imply in
both cases that $\deg(\Sres_{\ov k-1}(\frac{f}{x-\alpha},g))\le\ov
k-1<k$. Then
$$ \operatorname{Sylv}^{p,q}(A,B) (\alpha)
= (-1)^{c'+p}\,\cp_k\Big(-{\ov k\choose \ov q}G_{\ov
k-1}(\frac{f}{x-\alpha},g)g\Big)\,g(\alpha).$$
 When $\ov k-1<m-1 $, i.e $k\ge n$, we  apply Lemma~\ref{dos} and get
\begin{align*}
\operatorname{Sylv}^{p,q}(A,B) (\alpha)
&= (-1)^{c'+p}\,\Big( -{\ov k\choose \ov q}\,\cp_{k-n}\big(G_{\ov k-1}(\frac{f}{x-\alpha},g)\big)\,g(\alpha)\Big)\\
&= (-1)^{c'+p+k-n}\,\Big( -{\ov k\choose \ov q}G_{\ov
k}(f,g)(\alpha)g(\alpha)\Big)\\&= (-1)^{\ov p\,\ov q + n-p-1+nq}
\,\Big( -{\ov k\choose \ov q}G_{\ov
k}(f,g)(\alpha)g(\alpha)\Big).\end{align*} When $\ov k-1=m-1$,
$G_{\ov k-1}(\frac{f}{x-\alpha},g)=0=G_{\ov k}(f,g)$ and therefore
we also get
$$\operatorname{Sylv}^{p,q}(A,B) (\alpha)=(-1)^{\ov p \,\ov q+ n-p-1+nq} \,
\Big( -{\ov k\choose \ov q}G_{\ov k}(f,g)(\alpha)g(\alpha)\Big).$$

\smallskip\noindent
Analogously,
\begin{align*}
\operatorname{Sylv}^{p,q}(A,B) (\beta)& = (-1)^{q+\ov
p+c''}\,\cp_k\Big({\ov k \choose \ov p}F_{\ov
k-1}(f,\frac{g}{x-\beta})f -{\ov
k-1\choose \ov q-1}\Sres_{\ov k-1}(f,\frac{g}{x-\beta}))\Big)\,f(\beta)\\
&= (-1)^{q+\ov p+c''}\,\cp_k\Big({\ov k \choose \ov p}F_{\ov k-1}(f,\frac{g}{x-\beta})f\Big)\,f(\beta)\\
&= (-1)^{q+\ov p+c''}\,{\ov k \choose \ov p}\,\cp_{k-m}\big(F_{\ov
k-1}(f,\frac{g}{x-\beta})\big)\,f(\beta)\\ &= (-1)^{q+\ov
p+c''+1}\,{\ov k \choose \ov p}\,F_{\ov k}(f,g)(\beta) \,f(\beta),
\end{align*}
where $c''=\ov p\,(\ov q-1) +n-1-p -1+(n-1)q $.
Therefore,
$$\operatorname{Sylv}^{p,q}(A,B) (\beta)=(-1)^{\ov p \,\ov q+ n-p-1 +nq}\,{\ov k \choose
\ov p}\,F_{\ov k}(f,g)(\beta) \,f(\beta).$$
This concludes the proof.
\end{proof}

The next lemma covers  the cases $(p,n)$ and $(m,q)$ needed in the
proof of the previous result. Observe
 that \begin{align*}\operatorname{Sylv}^{p,n}(A,B)  &=
 g\,\sum_{ A^{\prime }\subset
A,
|A^{\prime}|=p}R(x,A^{\prime
})\,\frac{R(A^{\prime},B)}{R(A^{\prime},A- A^{\prime
})}\ \mbox{ for } \ p\le m,\\
\operatorname{Sylv}^{m,q}(A,B)  &=  f\, \sum_{B'\subset B, |B'|=q} R(x,B')\frac{R(A,B')}{R(B',B-B')} \ \mbox{ for } \ q\le n.\end{align*}

\begin{lemma}\label{teorema1} Set $1\le m\le n$. Then
\begin{enumerate}
\item \label{item1} $\operatorname{Sylv}^{p,n}(A,B) =  (-1)^{p}G_{\ov p-1}(f,g)\,g$ \
for $0\le p\le m-1$, i.e. $1\le \ov p \le m$.
\item \label{item4}  $\operatorname{Sylv}^{m,q}(A,B) =
(-1)^{n-m-1+nq}F_{\ov q-1}(f,g)\,f$ \  for $n-m-1\le q\le n-1$, i.e.
$1\le \ov q \le m+1$, when $m<n$ and  for $0\le q\le m-1$, i.e.
$1\le  \ov q \le m$,  when $m=n$.
\end{enumerate}
\end{lemma}

\begin{proof}
(\ref{item1}) By induction on  $m\ge 1$.
\\
The case $m=1$  is clear from Identities~\ref{m0} and \ref{Gm-1},
since in this case $p=0$ and $\ov{p}=1$.

\noindent Now set  $m>1$  and let $0\le p\le  m-1$. Both
$\operatorname{Sylv}^{p,n}(A,B)$ and $G_{\ov p-1 }(f,g)\,g$ are
polynomials of degree bounded by  $p+n<m+n$
 and we compare them by specializing them into the $m+n$ elements
$\alpha \in A$ and $\beta\in B$.
Clearly both expressions vanish at every $\beta\in B$ and so we only need to compare them at $\alpha\in A$.
\\
-- For $p<m-1$, we apply
Lemma~\ref{sylvester}, the
inductive hypothesis and Lemma~\ref{dos} (and the fact that $g$ is monic):
\begin{align*}\operatorname{Sylv}^{p,n}(A,B)(\alpha)&=(-1)^p\,\cp_{p+n}\big(\operatorname{Sylv}^{p,n}(A-\alpha,B)\big)\,g(\alpha)\\
&=(-1)^{2p}\,\cp_{p+n}\big(G_{(m-1)-p-1}(\frac{f}{x-\alpha},g)\,g\big)\,g(\alpha)\\
&= \cp_{p}\big(G_{(m-1)-p-1}(\frac{f}{x-\alpha},g)\big)\,g(\alpha)\
= \ (-1)^{p}\, G_{\ov p-1}(f,g)(\alpha)\,g(\alpha).
\end{align*}
-- For $p=m-1$:
\begin{align*}\operatorname{Sylv}^{p,n}(A,B)(\alpha)&=R(\alpha, A-\alpha)\frac{R(A-\alpha,B)}{R(A-\alpha,\alpha)}\,g(\alpha)\\
 &=(-1)^{m-1}\prod_{\alpha'\in A} g(\alpha')\ = \ (-1)^{m-1}\Res(f,g)\ = \ (-1)^{m-1}G_0(f,g)(\alpha)g(\alpha)
 ,\end{align*}
 by Identity~\eqref{result} and the fact that  $\Res(f,g)=F_0(f,g)f+G_0(f,g)g$ has degree $0$ in $x$.
Therefore $\operatorname{Sylv}^{p,n}(A,B)=(-1)^{p}G_{\ov p-1}(f,g)\,g$.

 \smallskip \noindent
 (\ref{item4}) By induction on  $n\ge m$.\\
 For $n=m$,  by Item (\ref{item1}) we have that for $0\le q\le m-1$,
 \begin{align*}
\operatorname{Sylv}^{m,q}(A,B)&=(-1)^{mq}\operatorname{Sylv}^{q,m}(B,A)\ = \ (-1)^{mq+q}G_{\ov q-1}(g,f)\,f\\&=(-1)^{mq+q}
 (-1)^{(m-(\ov q-1))(n-(\ov q-1))}F_{\ov q-1}(f,g)\,f \ = \
 (-1)^{nq-1}F_{\ov q-1}(f,g)\,f.\end{align*}
Now set $n\ge m+1$  and let $n-m-1\le q\le n-1$. Both
$\operatorname{Sylv}^{m,q}(A,B)$ and $F_{\ov q-1}(f,g)\,f$  are
polynomials of degree bounded by  $m+q<m+n$ and we compare them by
specializing them in the $m+n$ elements
$\alpha \in A$ and $\beta\in B$.  Clearly both expressions vanish at every $\alpha\in A$ and so we only need to compare them at $\beta\in B$. \\
-- For $q<n-1$, we apply
Lemma~\ref{sylvester}, the inductive hypothesis and Lemma~\ref{dos}:
\begin{align*}\operatorname{Sylv}^{m,q}(A,B)(\beta)&=(-1)^{q}\,\cp_{m+q}\big(\operatorname{Sylv}^{m,q}(A,B-\beta)\big)\,f(\beta)\\
&=(-1)^{q+(n-1-m-1)+(n-1)q}
\,\cp_{m+q}\big(F_{(n-1)-q-1}(f,\frac{g}{x-\beta})\,f\big)\,f(\beta)\\
&= (-1)^{(n+m-2+nq)+1}\, F_{\ov q-1}(f,g)(\beta)\,f(\beta).
\end{align*}
-- For $q= n-1$,
\begin{align*}\operatorname{Sylv}^{m,n-1}(A,B)(\beta)&=f(\beta)\,R(\beta, B-\beta)\frac{R(A,B-\beta)}{R(B-\beta,\beta)}\\
 &=(-1)^{n-1+m(n-1)}\prod_{\beta'\in B} f(\beta')\ = \ (-1)^{(mn+n-m-1)+mn}\Res(f,g)\\ &= \ (-1)^{n-m-1}F_0(f,g)(\beta)f(\beta).\end{align*}
Therefore $\operatorname{Sylv}^{m,q}(A,B)=(-1)^{n-m-1+nq}F_{\ov q-1}(f,g)\,f$ as wanted.\\
\end{proof}

As a particular case of Proposition~\ref{casogrande},  using
Identities~\eqref{Fm} and \eqref{Gm-1}, we obtain Case  (2) and a
particular case of Case (4) of the introduction:
\begin{corollary}\label{particular}{\ } \\ \begin{enumerate}\item
\label{k=m=n} Set $1\le m=n$ and let $ 0\le p, \ 0\le q$ be such
that $p+q=m$. Then  $$\operatorname{Sylv}^{p,q}(A,B)  =  {m
-1\choose q} \,f+{m-1\choose p} \,g.$$
\item \label{mm+2} Set $1\le m=n-2$  and let $0\le p\le m$, $0\le q$  be such that  $p+q=n-1$. Then $$\operatorname{Sylv}^{p,q}(A,B)=(-1)^{p+1}{m \choose p}\,f.$$
\end{enumerate}
\end{corollary}

This allows us to  simplify the rather long proofs for  the cases
when $p+q=m<n$, which appeared previously in \cite{LP},
\cite{DHKS07} and \cite{RoSp}.

\begin{proposition}\label{k=m<n} Set $1\le m\le n-1$ and let $p\ge 0, \ q\ge 0$ be such that $1\le p+q=m$. Then $$\operatorname{Sylv}^{p,q}(A,B)  =  {m \choose p} \,f.$$
\end{proposition}
\begin{proof} By induction on $n\ge m+1$, comparing the two expressions at the $n>m$ elements of $B$.\\
For $n=m+1$, by Lemma~\ref{sylvester} and
Corollary~\ref{particular}\eqref{k=m=n},
\begin{align*}\operatorname{Sylv}^{p,q}(A,B) (\beta)&= \cp_{m}\big(\operatorname{Sylv}^{p,q}(A,B-\beta)\big)\,f(\beta)\\
&=\cp_{m}\big({m -1\choose q} f+{m-1\choose p}
\,\frac{g}{x-\beta}\big) \, f(\beta)\\& =  \Big({m -1\choose q}
+{m-1\choose p}\Big)f(\beta)\ = \ {m \choose p}f(\beta) .
\end{align*}
Now set  $n>m+1$,
$$\operatorname{Sylv}^{p,q}(A,B) (\beta)\ =\ \cp_{m}\big(\operatorname{Sylv}^{p,q}(A,B-\beta)\big)\,f(\beta)
\ = \ \cp_{m}\big({m\choose p} f\big)\,f(\beta) \ =   \ {m \choose p}\,f(\beta) .
$$
\end{proof}

We finish the proof of Theorem~\ref{theorem} by splitting it into the two remaining cases to be proven. The first case is the inductive proof of \cite{RoSp}
that we repeat here for the sake of completeness.
\begin{proposition}\label{kchico} Set $1\le m\le n$ and let $p\ge 0$, $q\ge 0$ and $k=p+q$ be such that $k\le m$ when $m<n$ and  $k<m$ when $m=n$. Then
$$\operatorname{Sylv}^{p,q}(A,B)
=(-1)^{p(m-k)}{k\choose p}\Sres_k(f,g).$$
\end{proposition}
\begin{proof} By induction on $m\ge 1$:\\
The case $ m=1$  is completely covered by Identities~\eqref{00},
\eqref{m0}, \eqref{sresm0} and Proposition~\ref{k=m<n}.
\\
Now set $m>1$ and let $0\le k=p+q\le m$ if $m<n$ and $0\le k=p+q<m$ if $m=n$. We have\\
-- For $0\le k\le m-1$, we compare $\operatorname{Sylv}^{p,q}(A,B)$
and $\Sres_k(f,g) $, which are both of degree $k<m$, by specializing
them into the $m$ elements $\alpha \in A$ by means of
Lemma~\ref{sylvester}, the inductive hypothesis and
Corollary~\ref{corrs}:
\begin{align*}\operatorname{Sylv}^{p,q}(A,B)(\alpha)&=
(-1)^p\,\cp_k\big(\operatorname{Sylv}^{p,q}(A-\alpha,B)\big)\,g(\alpha)\\
&=(-1)^{p} (-1)^{p(m-1-k)}\,{k\choose
p}\,\cp_k\big(\Sres_{k}(\frac{f}{x-\alpha},g)\big) \,g(\alpha)\\ &=
(-1)^{p(m-k)}\,{k\choose p}\,\Sres_k(f,g)(\alpha).
\end{align*}
-- For $k= m<n$, it is Proposition~\ref{k=m<n}.
\end{proof}

\begin{proposition} Set $1\le m\le n-3$  and let $0\le p\le m$, $0\le q\le n$ be such that
$m+1\le p+q\le n-2$. Then $$\operatorname{Sylv}^{p,q}(A,B)=0.$$
\end{proposition}

\begin{proof} By induction on  $n\ge m+3$, specializing the expression  in the $n>m+1=k$ elements of $B$  by Lemma~\ref{sylvester}.\\
For $n=m+3$, by Corollary~\ref{particular}\eqref{mm+2}:
$$\operatorname{Sylv}^{p,q}(A,B)(\beta)  =  -\,f(\beta)\cp_{m+1}\big(\operatorname{Sylv}^{p,q}(A,B-\beta) \big)=-\,f(\beta)\cp_{m+1} \big((-1)^{p+1}{m \choose p}\,f\big)=0
,$$ since $\deg(f)=m<m+1$.\\
The case $n>m+3$  follows immediately.
\end{proof}

\bigskip
\begin{acknowledgement}
T.~Krick would like to thank the Mittag-Leffler Institute for hosting her in May 2011,  during the
preparation of this note.
\end{acknowledgement}


\bigskip
\begin{thebibliography}{}

\bibitem[{D'Andrea et al.(2007)}]{DHKS07}
D'Andrea, Carlos; Hong, Hoon; Krick, Teresa; Szanto, Agnes.
\newblock{\em An elementary proof of Sylvester's double sums for subresultants.\/}
\newblock  J. Symb. Comput. Vol. {\bf 42} (2007) 290--297.

\bibitem[{D'Andrea et al.(2009)}]{DHKS09}
D'Andrea, Carlos; Hong, Hoon; Krick, Teresa; Szanto, Agnes.
\newblock{\em  Sylvester's double sums: the general case.\/}
\newblock  J. Symbolic Comput.  Vol. {\bf 44} (2009) 1164--1175.

\bibitem[{Lascoux and Pragacz(2003)}]{LP}
Lascoux, Alain; Pragacz, Piotr. \newblock{\em Double Sylvester sums
for subresultants and multi-Schur functions.\/} \newblock J.
Symbolic Comput. 35 (2003),  no. 6, 689--710.

\bibitem[Roy and Szpirglas(2010)]{RoSp}
Roy, Marie-Fran\c coise; Szpirglas, Aviva.
 \newblock{\em Sylvester double sums and subresultants.\/} \newblock
 J. Symbolic Comput. (to appear) (2010).

\bibitem[{Sylvester(1853)}]{sylv}
Sylvester, James Joseph.
\newblock {\em On a theory of syzygetic relations of two rational integral
functions, comprising anapplication to the theory of Sturm's
function and that of the greatest algebraical common measure.\/}
\newblock Philosophical Transactions of the Royal Society of London, Part III
(1853), 407--548.
\newblock Appears also in Collected Mathematical Papers
          of James Joseph Sylvester, Vol 1,
\newblock Chelsea Publishing Co. (1973), 429-586.

\end{thebibliography}
\end{document}